\documentclass[titlepage,a4paper,11pt]{amsart} 
\usepackage[foot]{amsaddr}
\usepackage{amssymb}
\usepackage[lite]{amsrefs}
\usepackage{t1enc}
\usepackage[latin1]{inputenc}
\usepackage[english]{babel}
\usepackage{mathrsfs} 
\usepackage{hyperref}
\usepackage[all]{xy} 
\usepackage[lite]{amsrefs}
\usepackage[left=3cm,top=3cm,right=3cm,bottom=3cm]{geometry}

\newtheorem{thm}{Theorem}[section]
\newtheorem*{thm*}{Theorem}
\newtheorem{prop}[thm]{Proposition}

\newtheorem{rem}[thm]{Remark}

\DeclareMathOperator\F{\mathcal F}

\DeclareMathOperator\C{\mathbb{C}}
\DeclareMathOperator\R{\mathbb{R}}
\DeclareMathOperator\N{\mathbb{N}}

\DeclareMathOperator\E{\mathbb{E}}

\DeclareMathOperator\Span{span}
\DeclareMathOperator\id{\mathbf{1}}
\DeclareMathOperator\Tr{Tr}
\DeclareMathOperator\dett{det}
\DeclareMathOperator\sgn{sgn}
\DeclareMathOperator\alt{alt}
\DeclareMathOperator\reg{reg}
\DeclareMathOperator\GL{GL}

\hypersetup{
  colorlinks   = true, 
  urlcolor     = blue, 
  linkcolor    = blue, 
  citecolor   = red    
}
\title[HCIZ integral formula as unitarity of a canonical map]{HCIZ integral formula as unitarity of a canonical map between reproducing kernel spaces}
\date{\today}
\author{Martin Miglioli}
\email[Martin Miglioli]{martin.miglioli@gmail.com}
\address[Martin Miglioli]{Instituto Argentino de Matem\'atica-CONICET. Saavedra 15, Piso 3, (1083) Buenos Aires, Argentina}
\thanks{The author was supported by IAM-CONICET, grants PIP 2010-0757 (CONICET) and PICT 2010-2478 (ANPCyT)}

\begin{document}
\begin{abstract}
In this article we prove that the Harish-Chandra-Itzykson-Zuber (HCIZ) integral formula is equivalent to the unitarity of a canonical map between invariant subspaces of Segal-Bargmann spaces. As a consequence, we provide two new proofs of the HCIZ integral formula and alternative proofs of related results.\\

\medskip

\end{abstract}

\maketitle




\section{Introduction}
In \cite{harish} Harish-Chandra proved a formula for orbital integrals, see the expository article \cite{mcswiggen} and the references therein. The importance of such integrals for mathematical physics was first noted by Itzykson and Zuber \cite{itzykson}, the unitary integral is now known as the Harish-Chandra-Itzykson-Zuber (HCIZ) integral and has become an important identity in quantum field theory, random matrix theory, and algebraic combinatorics. It is usually written

\begin{align}\label{formulaoriginal}
\int_U e^{\Tr(uAu^{-1}B)}du=\left(\prod_{p=1}^{n-1}p!\right)\frac{\dett{\left[ e^{a_ib_j}\right] _{i,j=1}^n}}{\Delta(A)\Delta(B)}
\end{align}
where $\dett$ is the determinant of a matrix, $U$ is the group of $n$-by-$n$ unitary matrices, $A$ and $B$ are fixed $n$-by-$n$ diagonal matrices with eigenvalues $a_1< \dots <a_n$ and $b_1< \dots <b_n$ respectively, and
$$\Delta(A)=\prod_{i<j}(a_j-a_i)$$
is the Vandermonde determinant. In this article we link the HCIZ integral formula to the theory of Segal-Bargmann spaces thereby giving two new proofs of the integral formula. In Segal-Bargmann spaces \cite{segal,bargmann} holomorphic and reproducing kernel techniques are available. Also, the orthonormal bases and annihilation and creation operators have a simple form, see \cite{hall} and Chapter 4 of \cite{neretin}. Previous research on invariant subspaces of Segal-Bargmann spaces was done for example in \cite{rotinv} where rotationally invariant subspaces where studied. Also, in \cite[Section 6]{zhang} a unitarity result for the restriction to the Cartan algebra was given in the context of Dunkl theory. 

The paper is organized as follows. In Section \ref{secprel} we derive the reproducing kernels of invariant subspaces of two Segal-Bargmann spaces. We prove that the canonical map between these two spaces is unitary if and only if the HCIZ integral formula holds in Section \ref{secmain}. Finally, we provide alternative proofs of the HCIZ integral formula and other results in Section \ref{secalt}.

\section{Invariant subspaces of Segal-Bargmann spaces}\label{secprel}

In this section we recall results about Segal-Bargmann spaces and show properties of invariant subspaces of these spaces. References are Section 2, 3.2 and 6.1 of the lecture notes \cite{hall} and Chapter 4 of  \cite{neretin}. 

For $n\in \N$ let $\C^{n\times n}=M_n(\mathbb{C})$ be the space of $n\times n$ complex matrices with inner product $\langle x,y\rangle=\Tr(xy^*)$, where $\Tr$ is the trace. Let $U\subseteq M_n(\mathbb{C})$ be the group of unitary matrices and $\GL_n(\C)\subseteq M_n(\mathbb{C})$ the group of invertible matrices. The Segal-Bargmann space $\F(\C^{n\times n})$ is the space of holomorphic functions on $\C^{n\times n}$ such that 
$$\pi^{-n^2}\int_{\C^{n\times n}}\overline{F(z)}F(z)e^{-\Tr(z^*z)}dz<\infty$$
with inner product
$$\langle F,G\rangle = \pi^{-n^2}\int_{\C^{n\times n}}\overline{F(z)}G(z)e^{-\Tr(z^*z)}dz.$$
for $F,G\in\F(\C^{n\times n})$. For $a\in \C^{n\times n}$ the function 
$$K_a(z)=e^{\Tr(za^*)}$$ 
is called the coherent state with parameter $a$ and it satisfies
$$F(a)=\langle K_a,F\rangle$$
for all $F\in \F(\C^{n\times n})$. The function
$$K(z,a)=K_a(z)=e^{\Tr(za^*)}$$ 
is the reproducing kernel for the Segal-Bargmann space. We note that in $\F(\C^{n\times n})$ multiplication by a variable is the adjoint of derivation by the same variable. These are the creation and annihilation operators, see section 6.1 in \cite{hall}.

\begin{rem}
The Gaussian measure $\pi^{-n^2}e^{-\Tr(z^*z)}dz$ is the measure of the complex Ginibre ensemble, see \cite[Chapter 15]{mehta}. In this ensemble the expectation of the modulus squared of the sum of eigenvalues is
$$\E\vert\lambda_1(z)+\dots +\lambda_n(z)\vert^2=\langle \Tr(z),\Tr(z)\rangle =\left(\Tr\left(\frac{\partial}{\partial z}\right)\Tr(z)\right)\Bigg\vert_{z=0}=n,$$
where $\Tr(\frac{\partial}{\partial z})=\frac{\partial}{\partial z_{11}}+\cdots+\frac{\partial}{\partial z_{nn}}$. One can compute the inner product directly knowing that scaled monomials form an orthonormal basis, but we used the fact that multiplication by $z_{ij}$ is the adjoint of $\frac{\partial}{\partial z_{ij}}$. The expectation of the modulus squared of the product of eigenvalues is
$$\E\vert\lambda_1(z)\dots \lambda_n(z)\vert^2=\langle \dett(z),\dett(z)\rangle =\left(\dett\left(\frac{\partial}{\partial z}\right)\dett(z)\right)\Bigg\vert_{z=0}=n!.$$
\end{rem}

We define a unitary representation of $U$ on $\F(\C^{n\times n})$ by
$$u\mapsto (F(z)\mapsto F(u^{-1}zu))$$
for $u\in U$ and $F\in\F(\C^{n\times n})$. This representation is unitary since the Gaussian measure $\pi^{-n^2}e^{-\Tr(z^*z)}dz$ is invariant under conjugation by unitary matrices. The subspace of fixed points are the $F\in\F(\C^{n\times n})$ such that $F(z)= F(u^{-1}zu)$ for all $u\in U$:
$$\F(\C^{n\times n})^U=\{F\in\F(\C^{n\times n}): F(z)= F(u^{-1}zu)\mbox{ for all }u\in U\}.$$
We apply the ``unitarian trick'': for fixed $F\in\F(\C^{n\times n})^U$ and fixed $z\in \C^{n\times n}$ the map $\C^{n\times n}\to\C$ given by
$$x\mapsto F(z)- F(e^xze^{-x})$$
is holomorphic and vanishes for skew-Hermitian $x$, so it vanishes for all $x\in \C^{n\times n}$. Therefore, the functions in $\F(\C^{n\times n})^U$ are actually the functions invariant under conjugation by all $g\in \GL_n(\C)$, so they are functions of the spectrum. The orthogonal projection $P$ onto $\F(\C^{n\times n})^U$ is given by 
$$PF(x)=\int_U F(u^{-1}xu)du,$$ 
where the integral is over the Haar measure on $U$. The coherent states on $\F(\C^{n\times n})^U$ are given by 
$$Q_a=PK_a$$
since for $F\in \F(\C^{n\times n})^U$ and $a\in \C^{n\times n}$ we have 
$$F(a)=\langle K_a,F\rangle=\langle K_a,PF\rangle=\langle PK_a,F\rangle.$$
We denote by $Q: \C^{n\times n}\times\C^{n\times n}\to\C$ the reproducing kernel on $\F(\C^{n\times n})^U$. 

\begin{prop}\label{kernel1}
The reproducing kernel on $\F(\C^{n\times n})^U$ is given by
$$Q(z,a)=\int_U e^{\Tr(u^{-1}zua^*)}du$$
for $z,a\in \C^{n\times n}$.
\end{prop} 

\begin{proof}
Note that for $z,a\in \C^{n\times n}$
$$Q(z,a)=PK_a(z)=\int_U e^{\Tr(u^{-1}zua^*)}du.$$
\end{proof}

We now apply analogous arguments for another Segal-Bargmann space and another unitary representation. Let $\C^{n}$ be endowed with the dot product $x\cdot y=\sum_{i=1}^nx_iy_i$ and let $S_n$ be the symmetric group of degree $n$. The Segal-Bargmann space $\F(\C^{n})$ is the space of holomorphic functions on $\C^{n}$ such that 
$$\pi^{-n}\int_{\C^{n}}\overline{F(z)}F(z)e^{-\vert z\vert^2}dz<\infty$$
where $\vert z\vert^2=z\cdot\overline{z}$. The inner product is
$$\langle F,G\rangle = \pi^{-n}\int_{\C^{n}}\overline{F(z)}G(z)e^{-\vert z\vert^2}dz$$
for $F,G\in \F(\C^{n})$. The function
$$K(z,a)=K_a(z)=e^{z\cdot\overline{a}}$$ 
is the reproducing kernel for the Segal-Bargmann space. We define a unitary representation of $S_n$ on $\F(\C^{n})$ as
$$\sigma F(z_i)=\sgn(\sigma)F(z_{\sigma^{-1}(i)}),$$
where $\sgn(\sigma)$ denotes the sign of the permutation $\sigma$. The fixed point space is the space $\F(\C^{n})_{\alt}$ of alternating holomorphic functions:
$$\F(\C^{n})_{\alt}=\{F\in\F(\C^{n})): F(z_1,\dots,z_n)= \sgn(\sigma)F(z_{\sigma^{-1}(1)},\dots,z_{\sigma^{-1}(n)})\mbox{ for all }\sigma\in S_n\}.$$
The orthogonal projection $P$ onto this space is given by 
$$PF=\frac{1}{n!}\sum_{\sigma\in S_n}\sigma F.$$
The coherent states on $\F(\C^{n})_{\alt}$ are given by 
$$R_a=PK_a$$
since for $F\in \F(\C^{n})_{\alt}$ and $a\in \C^{n}$ we have 
$$F(a)=\langle K_a,F\rangle=\langle K_a,PF\rangle=\langle PK_a,F\rangle.$$
We denote by $R: \C^{n}\times\C^{n}\to\C$ the reproducing kernel on $\F(\C^{n})_{\alt}$.
\begin{prop}\label{kernel2}
The reproducing kernel on $\F(\C^{n})_{\alt}$ is given by
$$R(z,a)=\frac{1}{n!}\sum_{\sigma\in S_n}\sgn(\sigma)e^{\sigma(z)\cdot\overline{a}}.$$
\end{prop}

\begin{proof}
Note that for $z,a\in \C^{n}$
$$R(z,a)=PK_a(z)=\frac{1}{n!}\sum_{\sigma\in S_n}\sigma K_a(z)=\frac{1}{n!}\sum_{\sigma\in S_n}\sgn(\sigma)e^{\sigma(z)\cdot\overline{a}}.$$
\end{proof}

\section{HCIZ integral formula as unitarity of a canonical map}\label{secmain}

In this section we prove the main result of the article. The next proposition gives a formulation of the HCIZ integral formula in terms of inner products in $\F(\C^{n\times n})^U$ and in $\F(\C^{n})_{\alt}$. We denote with $D\subseteq \C^{n\times n}$ the set of complex diagonal matrices, with $D_{\reg}$ the set of complex diagonal matrices with distinct eigenvalues, and with $D_{\R}$ the set of diagonal matrices with real entries. As before $\Delta$ denotes the Vandermonde determinant. We set the constant
$$c=\left(\prod_{p=1}^{n}p!\right)^{-\frac{1}{2}}.$$ 

\begin{prop}\label{propigauldaddeprodint}
The HCIZ integral formula is equivalent to
\begin{equation}\label{formprodint}
\left\langle Q_x,Q_y\right\rangle_{\F(\C^{n\times n})^U}=\left\langle \frac{R_x}{c\Delta(\overline{x})}, \frac{R_y}{c\Delta(\overline{y})}\right\rangle_{\F(\C^{n})_{\alt}}
\end{equation}
for all $x,y\in D_{\reg}$.
\end{prop}

\begin{proof}
The left hand side is
$$\left\langle Q_x,Q_y\right\rangle=Q_y(x)=\int_U e^{\Tr(u^{-1}xuy^*)}du,$$
where we used the reproducing property and the formula for the kernel given in Proposition \ref{kernel1}. The right hand side is
\begin{align*}
\left\langle \frac{R_x}{c\Delta(\overline{x})}, \frac{R_y}{c\Delta(\overline{y})}\right\rangle
&=\frac{1}{c^2\Delta(x)\Delta(\overline{y})}R_y(x)\\
&=\left(\prod_{p=1}^{n}p!\right)\frac{1}{\Delta(x)\Delta(\overline{y})}R_y(x)\\
&=\left(\prod_{p=1}^{n}p!\right)\frac{1}{\Delta(x)\Delta(\overline{y})}\frac{1}{n!}\sum_{\sigma\in S_n}\sgn(\sigma)e^{\sigma(x)\cdot\overline{y}},
\end{align*}
where in the third equality we used the formula for the reproducing kernel given in Proposition \ref{kernel2}. The HCIZ formula (\ref{formulaoriginal}) states that the left hand side and the right hand side agree when $x,y\in D_{\reg}\cap D_{\R}$. Since both sides are holomorphic in $x$ and anti-holomorphic in $y$ we get equality for all $x,y\in D_{\reg}$. The converse is straightforward. 
\end{proof}

\begin{thm}\label{teoprincipal}
The HCIZ integral formula implies that the map $\psi:\F(\C^{n\times n})^U\to\F(\C^{n})_{\alt}$ given by
$$\psi(F)(x)=c\Delta(x) F\vert_{D}(x)$$
for $F\in\F(\C^{n\times n})^U$ and $x\in D$ is well defined and unitary. Conversely, if this map is well defined and unitary, then the HCIZ integral formula holds. The map $\psi$ satisfies 
$$\psi(Q_x)=\frac{R_x}{c\Delta(\overline{x})}$$
for any $x\in D_{\reg}$.
\end{thm}

\begin{proof}
Define the map 
$$\phi:\Span\{Q_x\}_{x\in D_{\reg}}\to \Span\left\{\frac{R_x}{c\Delta(\overline{x})}\right\}_{x\in D_{\reg}}$$
by
$$\sum_{i=1}^m\alpha_iQ_{x_i}\mapsto \sum_{i=1}^m\alpha_i\frac{R_{x_i}}{c\Delta(\overline{x_i})},$$
where $\alpha_i\in\C$. Equation (\ref{formprodint}) in Proposition \ref{propigauldaddeprodint} implies that the map is well defined and an isometry, therefore it extends to an isometry between 
$$\overline{\Span}\{Q_x\}_{x\in D_{\reg}}\quad\mbox{ and }\quad \overline{\Span}\left\{\frac{R_x}{c\Delta(\overline{x})}\right\}_{x\in D_{\reg}},$$
where $\overline{\Span}(A)$ denotes the closure of $\Span(A)$. 

We now prove that $\overline{\Span}\{Q_x\}_{x\in D_{\reg}}=\F(\C^{n\times n})^U$. If this does not hold take an $F\in \F(\C^{n\times n})^U$ which is orthogonal to $\Span\{Q_x\}_{x\in D_{\reg}}$. This means that 
$$F(x)=\langle Q_x,F\rangle=0$$
for all $x\in D_{\reg}$. Since $F$ is invariant by conjugation of the variables it vanishes on all $g\in \GL_n(\C)$ with $n$ distinct eigenvalues. Since these matrices are dense in $\C^{n\times n}$ we conclude that $F=0$. The fact that 
$$\overline{\Span}\left\{\frac{R_x}{c\Delta(\overline{x})}\right\}_{x\in D_{\reg}}=\F(\C^{n})_{\alt}$$
is proved similarly. Therefore $\phi$ defines a unitary map from $\F(\C^{n\times n})^U$ onto $\F(\C^{n})_{\alt}$. For $F\in\F(\C^{n\times n})^U$ and $x\in D_{\reg}$ we get
\begin{align*}
F(x)=&\langle Q_x,F\rangle=\langle \phi(Q_x),\phi(F)\rangle=\left\langle\frac{R_x}{c\Delta(\overline{x})},\phi(F) \right\rangle\\
&=\frac{1}{c\Delta(x)}\langle R_x,\phi(F)\rangle=\frac{1}{c\Delta(x)}\phi(F)(x),
\end{align*}
where the first and last equalities follow from the reproducing property, the second from the unitarity of $\phi$, and the third from the definition of $\phi$. Therefore 
$$\phi(F)(x)=c\Delta(x)F(x)$$
for $x\in D_{\reg}$. Hence, the map $\phi$ is equal to the map $\psi$ and the first assertion of the theorem is proved.

To prove the second assertion assume that $\psi$ is well defined and unitary. Then for $F\in\F(\C^{n\times n})^U$ and $x\in D_{\reg}$ we have 
$$F(x)=\langle Q_x,F\rangle=\langle \psi(Q_x), \psi(F)\rangle=\langle \psi(Q_x),c\Delta F\vert_D\rangle.$$  
Also 
$$\langle R_x, \psi(F)\rangle=\langle R_x, c\Delta F\vert_D\rangle=c\Delta(x)F(x),$$
hence 
$$F(x)=\left\langle \frac{R_x}{c\Delta(\overline{x})}, \psi(F)\right\rangle.$$
Therefore for fixed $x\in D_{\reg}$
$$\langle \psi(Q_x), \psi(F)\rangle=\left\langle \frac{R_x}{c\Delta(\overline{x})}, \psi(F)\right\rangle$$
for all $F\in \F(\C^{n\times n})^U$, and since $\psi$ is unitary this implies that 
$$\psi(Q_x)=\frac{R_x}{c\Delta(\overline{x})}.$$
From this property and Proposition \ref{propigauldaddeprodint} the HCIZ formula follows.
\end{proof}

\begin{rem}
Note that multiplication by $\Delta$ provides a linear isomorphism between symmetric and alternating polynomials on $\C^n$. Therefore, Theorem \ref{teoprincipal} implies that the restriction map $F\mapsto F\vert_D$ is a linear isomorphism between conjugation invariant polynomials on $\C^{n\times n}$ and symmetric polynomials on $\C^n$. This is a psecial case of the Chevalley restriction theorem, see \cite{chevalley}.

\end{rem}

\section{Differentiation of polynomials and orthonormal bases}\label{secalt}

In this section we give alternative proofs of the HCIZ integral and other results. For a polynomial $F$ we define $F^*(z)=\overline{F(\overline{z})}$, that is, the coefficients of $F^*$ are the complex conjugates of the coefficients of $F$. An alternative proof of the HCIZ integral formula is given by the following

\begin{prop}
The unitarity of the map $\psi$ in Theorem \ref{teoprincipal} implies that for polynomials $F,G\in F\in \F(\C^{n\times n})^U$ the equation
\begin{align}\label{formuladif}
\Delta(x).F\left(\frac{\partial}{\partial z}\right)G(x)=\left(F\Big\vert_D\left(\frac{\partial}{\partial z_1},\dots,\frac{\partial}{\partial z_n}\right)(\Delta. G\vert_D)\right) (x)
\end{align}
holds for all $x\in D$. Conversely, if formula (\ref{formuladif}) holds for all polynomials $F,G\in \F(\C^{n\times n})^U$, then the map $\psi$ is unitary.
\end{prop}

\begin{proof}
We denote by $M_FG=F.G$ the multiplication operator. For a polynomial $F\in \F(\C^{n\times n})^U$ it is easy to check that 
$$M_F=\psi^{-1}\circ M_{F\vert_D}\circ \psi.$$
Therefore, by applying adjoints 
$$F^*\left(\frac{\partial}{\partial z}\right)=(M_F)^*=\psi^{-1}\circ (M_{F\vert_D})^*\circ \psi=\psi^{-1}\circ F^*\Big\vert_D\left(\frac{\partial}{\partial z_1},\dots,\frac{\partial}{\partial z_n}\right)\circ \psi.$$
Here we used the unitarity of $\psi$ and the fact that multiplication by a variable is the adjoint of derivation by the same variable. We evaluate
$$ \left(F^*\Big\vert_D\left(\frac{\partial}{\partial z_1},\dots,\frac{\partial}{\partial z_n}\right)\circ \psi\right)G=cF^*\Big\vert_D\left(\frac{\partial}{\partial z_1},\dots,\frac{\partial}{\partial z_n}\right)(\Delta.G\vert_D).$$
Applying $\psi^{-1}$ and evaluating at $x\in D_{\reg}$ is the same as multiplying by
$$\frac{1}{c\Delta(x)},$$
so we get the formula for $F^*$. Since $F=(F^*)^*$ the first claim follows.

We denote for simplicity $\partial=\frac{\partial}{\partial z}$. To prove the second claim we need to verify that for all polynomials $F,G\in F\in \F(\C^{n\times n})^U$ the equality
$$\langle F^*,G\rangle=\langle \psi(F^*),\psi(G)\rangle=\langle c\Delta F^*\vert_D,c\Delta G\vert_D\rangle$$
holds. By the definition of the inner product in terms of differentiation at $z=0$ we have 
$$\langle F^*,G\rangle=F(\partial)G\vert_{z=0}$$
and
$$\langle c\Delta F^*\vert_D,c\Delta G\vert_D\rangle=c^2\Delta(\partial)F\vert_D(\partial)(\Delta G\vert_D)\vert_{z=0}.$$
Therefore, we need to verify that 
$$F(\partial)G\vert_{z=0}=c^2\Delta(\partial)F\vert_D(\partial)(\Delta G\vert_D)\vert_{z=0},$$
which by formula (\ref{formuladif}) is equivalent to
$$\frac{1}{\Delta} F\vert_D(\partial)(\Delta G\vert_D)\Big\vert_{z=0}=c^2\Delta(\partial)F\vert_D(\partial)(\Delta G\vert_D)\vert_{z=0}.$$
We set 
$$F\vert_D(\partial)(\Delta G\vert_D)=\Delta(d+H),$$
where $d\in\C$ and $H$ is a symmetric polynomial on $\C^n$ without constant term.

Hence, we have to show that
$$\frac{1}{\Delta}\Delta(d+H)\Big\vert_{z=0}=c^2\Delta(\partial)(\Delta(d+H))\vert_{z=0}.$$
The left hand side is $d$ and the right hand side is
$$c^2d\Delta(\partial)\Delta\vert_{z=0}+c^2\Delta(\partial)(\Delta.H)\vert_{z=0}.$$
We have 
$$\Delta(\partial)\Delta\vert_{z=0}=\left(\prod_{p=1}^{n}p!\right)=\frac{1}{c^2}.$$
Since $\Delta$ is homogeneous of degree $(n-1)!$ and $\Delta.H$ is a sum of monomials of higher degree we get
$$\Delta(\partial)(\Delta.H)\vert_{z=0}=0.$$
Therefore the right hand side is also $d$.
\end{proof}

Formula (\ref{formuladif}) was obtained by Harish-Chandra in \cite{harish} for semi-simple Lie groups. His proof of (\ref{formulaoriginal}) is based on this result, see \cite{harish} and Theorem 3.8 in \cite{mcswiggen}. 

\begin{rem}
If $F\in  \F(\C^{n\times n})^U$ is a polynomial then
$$F\left(\frac{\partial}{\partial z}\right)Q_a=F(\overline{a})Q_a$$
for any $a\in\C^{n\times n}$. This follows from 
$$\left\langle F\left(\frac{\partial}{\partial z}\right)Q_a,G\right\rangle=\langle Q_a,F^*.G\rangle=F^*(a)G(a)=F^*(a)\langle Q_a,G\rangle=\langle\overline{F^*(a)}Q_a,G\rangle$$
for any polynomial $G\in  \F(\C^{n\times n})^U$. By the same argument
$$F\left(\frac{\partial}{\partial z}\right) R_a=F(\overline{a})R_a$$
for any $a\in\C^n$ and any symmetric polynomial $F$ on $\C^n$. 
\end{rem}

Let $\N_0$ stand for the non negative integers. For $\mu\in \N_0^n$ we denote the monomials as usual with $z^{\mu}=z_1^{\mu_1}\dots z_n^{\mu_n}$ and we use the notation $\mu!=\mu_1!\dots\mu_n!$.  In the Segal-Bargmann space $\F(\C^n)$ an orthonormal basis of the space is
$$\left(\frac{1}{\sqrt{\mu!}}z^{\mu}\right)_{\mu\in \N_0^n},$$
see \cite[Section 3.2]{hall}. The set $\Pi$ of partitions is defined as
$$\Pi=\{\lambda\in \N_0^n:\lambda_1\geq\lambda_2\geq\dots\geq\lambda_n\}.$$
We set $\delta=(n-1,n-2,\dots,0)$ and $\id=(1,1,\dots,1)$. For $\lambda\in\Pi$, if $(x_1,\dots,x_n)$ are the eigenvalues of $x\in \GL_n(\C)$ we define 
$$\chi_{\lambda}(x) = s_{\lambda}(x_1,\dots,x_n),$$ 
where $s_{\lambda}$ is a Schur polynomial. These polynomials are defined by
$$s_{\lambda}(x_1,\dots,x_n)=\frac{a_{\lambda+\delta}(x_1,\dots,x_n)}{a_{\delta}(x_1,\dots,x_n)},$$
where 
$$a_{\mu}(x_1,\dots,x_n)=\dett\left[x_i^{\mu_j}\right]_{i,j=1}^n.$$
Note that $a_{\delta}(x_1,\dots,x_n)=\Delta(x_1,\dots,x_n)$ is the Vandermonde determinant.

\begin{rem}
The irreducible polynomial representations of the general linear group $\GL_n(\C)$ are labelled by Young diagrams, which we think of as vectors $\lambda \in\Pi$. The character of the $\lambda$-representation is given by $\chi_{\lambda}$.
\end{rem}

For $\lambda\in\Pi$ we denote
$$d_{\lambda}(z)=\frac{1}{\sqrt{n!(\lambda+\delta)!}} a_{\lambda+\delta}(z)\in\F(\C^n)_{\alt}.$$

\begin{prop}\label{ortonalt}
An orthonormal basis of the space $\F(\C^n)_{\alt}$ is given by
$$ \left(d_{\lambda}\right)_{\lambda\in\Pi}.$$ 
\end{prop}

\begin{proof}
It is straightforward to check that this set is an orthonormal set. To check that it is a basis assume that $F\in\F(\C^n)_{\alt}$ and that $\langle F,a_{\lambda+\delta}\rangle=0$ for all $\lambda\in\Pi$. Take a $\mu\in\N_0^n$ and assume that all the $\mu_i$ are different. Then $z^{\mu}=\sigma(z^{\lambda+\delta})$ for a $\sigma\in S_n$ and a $\lambda\in\Pi$. We consider as in Section \ref{secprel} the orthogonal projection $P$ onto $\F(\C^n)_{\alt}$. Then 
$$\langle F,z^{\mu}\rangle=\langle PF,z^{\mu}\rangle=\langle F,Pz^{\mu}\rangle=\left\langle F,\sgn(\sigma)\frac{1}{n!}a_{\lambda+\delta}\right\rangle=0.$$
If there are $i\neq j$ such that $\mu_i=\mu_j$ take as $\sigma$ the transposition of $i$ and $j$. Then 
$$\langle F,z^{\mu}\rangle=\langle \sigma(F),z^{\mu}\rangle=\langle F,\sigma(z^{\mu})\rangle=-\langle F,z^{\mu}\rangle,$$
so that $\langle F,z^{\mu}\rangle=0$. We proved that the inner product of $F$ with all the elements of the orthonormal basis of $\F(\C^n)$ vanish, so $F=0$.
\end{proof}

For $\lambda\in\Pi$ we denote
$$e_{\lambda}(z)=\sqrt{\frac{\delta !}{(\lambda+\delta)!}} \chi_{\lambda}(z)\in\F(\C^{n\times n})^{U}.$$
Another proof of the HCIZ integral formula is given by

\begin{prop}
The unitarity of the map $\psi$ in Theorem \ref{teoprincipal} implies that $(e_{\lambda})_{\lambda\in\Pi}$ is an orthonormal basis of $\F(\C^{n\times n})^{U}$. Conversely, the fact that $(e_{\lambda})_{\lambda\in\Pi}$ is an orthonormal basis implies the unitarity of $\psi$.
\end{prop}

\begin{proof}
We check that for $\lambda \in \Pi$ 
$$\psi(e_{\lambda})=ca_{\delta}e_{\lambda}\vert_D=d_{\lambda}.$$
The proposition follows.
\end{proof}

The fact that $(e_{\lambda})_{\lambda\in \Pi}$ is orthonormal was proved in \cite[Proposition 2]{rains}. Since $\F(\C^{n\times n})^{U}$ and $\F(\C^{n})_{\alt}$ are reproducing kernel spaces their kernels can be diagonalized in terms of orthonormal bases. Proposition 1.6 in \cite{hilgert} yields

\begin{prop}
The kernel of $\F(\C^{n})_{\alt}$ can be written as
$$R(x,y)=\sum_{\lambda\in \Pi}d_{\lambda}(x)\overline{d_{\lambda}(y)},$$
where the sum converges absolutely and uniformly on compact subsets of $\C^{n}\times\C^{n}$.
The kernel of $\F(\C^{n\times n})^{U}$ can be written as
 
$$Q(x,y)=\sum_{\lambda\in \Pi}e_{\lambda}(x)\overline{e_{\lambda}(y)},$$
that is 
$$\int_U e^{\Tr(u^{-1}xuy^*)}du=\sum_{\lambda\in \Pi}\frac{\delta !}{(\lambda+\delta)!} \chi_{\lambda}(x)\overline{\chi_{\lambda}(y)},$$
where the sum converges absolutely and uniformly on compact subsets of $\C^{n\times n}\times\C^{n\times n}$.
\end{prop}  

A partial converse to this proposition is

\begin{prop}
If the expansion 
$$Q(x,y)=\sum_{\lambda\in \Pi}e_{\lambda}(x)\overline{e_{\lambda}(y)}$$
holds, then $(e_{\lambda})_{\lambda\in \Pi}$ is an orthonormal basis of $\F(\C^{n\times n})^{U}$.
\end{prop}

\begin{proof}
The $e_{\lambda}$ with $\vert \lambda\vert=\lambda_1+\dots+\lambda_n=m\in\N_0$ are polynomials of degree $m$ and are orthogonal to $e_{\lambda'}$ with $\vert \lambda'\vert=m'\neq m$. Therefore, for $\lambda'$ such that $\vert\lambda'\vert=m\in\N_0$ we have
$$e_{\lambda'}(a)=\langle Q(z,a),e_{\lambda'}(z)\rangle=e_{\lambda'}(a)\langle e_{\lambda'},e_{\lambda'}\rangle + \sum_{\lambda:\vert\lambda\vert=m,\lambda\neq\lambda'}e_{\lambda}(a)\langle e_{\lambda},e_{\lambda'}\rangle.$$
For a $\lambda''$ with $\vert \lambda''\vert=m$ such that $\lambda''\neq\lambda'$ choose an $a\in D$ such that $e_{\lambda''}(a)\neq 0$ and $e_{\lambda}(a)= 0$ for all other $\lambda$ with $\vert\lambda\vert=m$. We conclude that
$$e_{\lambda''}(a)=e_{\lambda''}(a)\langle e_{\lambda''},e_{\lambda'}\rangle=0,$$
so the orthonormality of $(e_{\lambda})_{\lambda\in \Pi}$ follows. To prove that this set is a basis let $F\in\F(\C^{n\times n})^{U}$ be written as a sum $F=\sum_{m\in\N_0}P_m$ of polynomials $P_m$ of degree $m$. If $\langle e_{\lambda},F\rangle=0$ for all $\lambda\in\Pi$, then for $m\in\N_0$ we have $\langle e_{\lambda},P_m\rangle=0$ for all $\lambda$ with $\vert\lambda\vert=m$. Therefore, for $a\in D$ we have
$$P_m(a)=\langle Q_a,P_m\rangle=\left\langle\sum_{\lambda:\vert\lambda\vert=m}\overline{e_{\lambda}(a)}e_{\lambda}(z),P_m(z)\right\rangle=0,$$
so all the $P_m$ vanish and $F=0$.
\end{proof}

The next proposition computes the coefficients of the expansion of an invariant holomorphic function in terms of the characters $\chi_{\lambda}$.

\begin{prop}
For an $F\in\F(\C^{n\times n})^U$ written as
$$F=\sum_{\lambda\in\Pi}f_{\lambda}\chi_{\lambda}$$
the coefficients $f_{\lambda}$ are given by
$$f_{\lambda}=\left(\mbox{  coefficient of  }z^{\lambda+\delta}\mbox{  in  } \Delta.F\vert_D\right)$$
for $\lambda\in\Pi$.
\end{prop}

\begin{proof}
We have the expansion of $F$ in terms of the orthonormal basis
$$F=\sum_{\lambda\in\Pi}\langle e_{\lambda},F\rangle e_{\lambda}.$$
By the unitarity of $\psi$ the Fourier coefficients are
$$\langle e_{\lambda},F\rangle=\langle \psi(e_{\lambda}),\psi(F)\rangle=\langle d_{\lambda},c\Delta.F\vert_D\rangle$$
for $\lambda\in\Pi$. Note that 
$$d_{\lambda}(z)=\frac{1}{\sqrt{n!(\lambda+\delta)!}} a_{\lambda+\delta}(z)=\frac{1}{\sqrt{n!(\lambda+\delta)!}}n!P(z^{\lambda+\delta})=\sqrt{\frac{n!}{(\lambda+\delta)!}}P(z^{\lambda+\delta}),$$
where $P$ is the orthogonal projection to the alternating functions as in Section \ref{secprel}. Therefore
$$\langle d_{\lambda},c\Delta.F\vert_D\rangle=\sqrt{\frac{n!}{(\lambda+\delta)!}}\left\langle P(z^{\lambda+\delta}),c\Delta.F\vert_D\right\rangle=\sqrt{\frac{n!}{(\lambda+\delta)!}}\left\langle z^{\lambda+\delta},c\Delta.F\vert_D\right\rangle.$$
Also $\langle z^{\lambda+\delta},c\Delta.F\vert_D\rangle=(\lambda+\delta)!(\mbox{  coefficient of  }z^{\lambda+\delta}\mbox{  in  } c\Delta.F\vert_D)$. Some calculations yield
$$f_{\lambda}=\sqrt{\prod_{p=1}^n p!}\left(\mbox{  coefficient of  }z^{\lambda+\delta}\mbox{  in  } c\Delta.F\vert_D\right),$$
and since $\sqrt{\prod_{p=1}^n p!}=\frac{1}{c}$ the proposition follows.
\end{proof}

\section*{Acknowledgements}
I thank Colin McSwiggen for several discussions related to the Harish-Chandra-Itzykson-Zuber integral and I thank K.-H. Neeb for comments on the manuscript.






\noindent
\end{document}